\documentclass[12pt]{amsart}
\usepackage{graphicx} 
\usepackage[utf8]{inputenc}
\usepackage{geometry}
\usepackage{amssymb}
\usepackage{indentfirst,amscd}
\usepackage{amsthm, amsfonts, empheq}
\usepackage{mathtools}
\usepackage{enumitem}
\usepackage{tikz-cd} 
\usepackage{adjustbox}
\usepackage[style=alphabetic, giveninits=true]{biblatex}
\usepackage{thmtools}
\usepackage{thm-restate}
\usepackage{hyperref}

\theoremstyle{definition}
\newtheorem{Def}{Definition}[section]

\theoremstyle{plain}
\declaretheorem[name=Theorem, numberwithin=section]{Theo}
\declaretheorem[name=Lemma, numberwithin=section]{Lemma}
\declaretheorem[name=Proposition, numberwithin=section]{Prop}
\declaretheorem[name=Corollary, numberwithin=section]{Cor}

\theoremstyle{remark}
\newtheorem{Rem}{Remark}[section]
\newtheorem{Example}{Example}[section]

\newcommand{\g}{\mathfrak{g}}
\newcommand{\h}{\mathfrak{h}}
\newcommand{\dtheta}{d_{\theta}}
\newcommand{\fg}{\pi_1}

\DeclareMathOperator{\Tr}{Tr}
\DeclareMathOperator{\Aut}{Aut}
\DeclareMathOperator{\Ric}{Ric}
\DeclareMathOperator{\Scal}{Scal}
\DeclareMathOperator{\Id}{Id}

\DeclareMathOperator{\diag}{diag}
\newcommand{\Z}{\mathbb{Z}}
\newcommand{\R}{\mathbb{R}}
\newcommand{\N}{\mathbb{N}}

\title[Vaisman Solvmanifolds as Finite Quotients]{Vaisman Solvmanifolds as Finite Quotients of Kodaira-Thurston Nilmanifolds}
\author[Lucas H. S. Gomes]{Lucas H. S. Gomes}
\address{Department of Mathematics, Graduate School of Science, The University of Osaka, Osaka, Japan}
\email{u788855c@ecs.osaka-u.ac.jp}

\begin{document}
\begin{abstract}
    We prove that every Vaisman solvmanifold is a finite quotient of a Kodaira–Thurston manifold. More generally, we show that every aspherical compact Vaisman manifold with strongly polycyclic fundamental group is a finite quotient of a Kodaira–Thurston manifold. As consequences, we obtain that every completely solvable solvmanifold admitting a Vaisman structure is a Kodaira–Thurston manifold, that Oeljeklaus–Toma manifolds admit no Vaisman structures (not necessarily left-invariant), and that solvmanifolds does not admit LCK Einstein–Weyl structures.
\end{abstract}
\maketitle

\tableofcontents

\section{Introduction}

Let $G$ be a Lie group and $\Gamma \leq G$ be a lattice, i.e., $\Gamma$ is a discrete cocompact subgroup of $G$. The quotient $\Gamma \backslash G$ is called a solvmanifold (nilmanifold) if $G$ is a solvable (nilpotent) Lie group. A Kodaira-Thurston manifold is a nilmanifold with associated Lie group $\R \times H(1,n)$, where $H(1,n)$ is the Heisenberg group of dimension $2n+1$ (Definition \ref{HeinsbrgDef}). 

A classical problem in the literature is to classify geometric structures in solvmanifolds. One can study the case where the structure is left-invariant with respect to the associated Lie group $G$, or work with the general case when the structure is not necessarily left-invariant. Examples of such approach are given by Hasegawa's Theorem in \cite{Hasegawa} where he shows that every solvmanifold that admits a K\"ahler structure is the finite quotient of a complex torus. Another example is given by Kasuya in \cite{Kasuya1} where he shows that every solvmanifold that admits a Sasakian structure is a finite quotient of a Heisenberg nilmanifold, that is, $G$ is the Heisenberg group $H(1,n)$. Both results rely on the properties of the fundamental group of the solvmanifold. In this work, we focus our attention to the Locally Conformal K\"ahler case.

Let $(M^n, J, \omega)$ be a Hermitian manifold with complex structure $J$ and fundamental form $\omega$ with $n \geq 4$. 
If there exists a closed 1-form $\theta$ such that $d\omega = \theta \wedge \omega$, then the triple $(J, \omega, \theta)$ is called a Locally Conformally K\"ahler (LCK) structure on $M$. 
We call this structure Vaisman if $\theta$ is parallel with respect to the Levi-Civita connection, $\nabla^g \theta = 0$.

For the general classification of LCK structures on solvmanifolds, we can consider a combination of the cases where some element of the triple is or is not left-invariant. Initially, Ugarte in \cite{UgarteMain} showed that, in real dimension at most 6, every nilmanifold with an LCK structure whose complex structure $J$ is left-invariant has its associated Lie group $G$ isomorphic to $\R \times H(1,n)$. Ugarte then conjectured that such result should also be true in any dimension, which was proven by Sawai in \cite{Sawai1}. Sawai later extended the result in \cite{Sawai2} for the Vaisman case, showing that a completely solvable solvmanifold with a Vaisman structure whose $J$ is left-invariant has its associated Lie group $G$ isomorphic to $\R \times H(1,n)$.

At this point, one could hope that for $G$ nilpotent, even if $J$ is not left-invariant it would still be true that $G$ should be isomorphic to $\R \times H(1,n)$. Bazzoni showed in \cite{BazzoniMain} that this is true when the structure is Vaisman. For the case where the structure is in general only LCK, the problem is still open. 

When $G$ is solvable, even if $(J, \omega, \theta)$ is left-invariant and Vaisman, it has been proved that $G$ is not isomorphic to $\R\times H(1,n)$ in general. Andrada and Origlia in \cite{AndradaOriglia} completely characterized the solvmanifolds in which the LCK structure $(J, \omega, \theta)$ is Vaisman and left-invariant. They used such characterization to construct a family of examples in which $G$ was not isomorphic to $\R\times H(1,n)$. 

Inspired by the previously mentioned works of Hasegawa and Kasuya, we prove an analogue of their results for the case of Vaisman solvmanifolds, extending the works of Bazzoni and Sawai for the case where $J$ is not left-invariant, which is the main result of this paper.

\begin{restatable*}{Theo}{MainTheo}\label{MainTheorem}
	Let $M = \Gamma\backslash G$ be a solvmanifold. If $M$ admits a Vaisman structure $(J, \omega, \theta)$, then $M$ is diffeomorphic to a finite quotient of a Kodaira-Thurston manifold $N = (\Z \times \Lambda)\backslash(\R \times H_{2n+1})$, for some $\Lambda \leq \Gamma$. 
	
	More precisely, we have a smooth finite covering $p : N \to M$ in the following diagram.
	\[
	\begin{tikzcd}[ampersand replacement=\&]
		G \arrow[rd] \arrow[rdd, "\pi"', bend right] \&                                               \& {\mathbb{R}\times H_{2n+1}} \arrow[ld] \\
		\&  N \arrow[d, "p"] \&                                           \\
		\& \Gamma \backslash G                                    \&                                          
	\end{tikzcd}
	\]
\end{restatable*}  

We are thus able to directly generalize Bazzoni's result in \cite{BazzoniMain} to the completely solvable case.

\begin{restatable*}{Cor}{ComplSolv}
	Let $M = \Gamma \backslash G$ be a completely solvable solvmanifold admitting a Vaisman structure $(J, \omega, \theta)$. Then $G$ is isomorphic to $\R \times H(1,n)$.
\end{restatable*}

More generally, one can extend the main result to the case in which $M$ admits a LCK structure with potential. This is easily done, because the structure is not fixed in the proof of Theorem \ref{MainTheorem}, all one needs is that $M$ admits a Vaisman structure, and by a deformation result of LCK structure with potential in \cite{OrneaVerbitsky}, one can deform the structure to obtain a Vaisman structure. The details are given in section 4. The proof also doesn't depend strongly on the fact that $M$ has the structure of a solvmanifold, except for the fact that $M$ is, in particular, aspherical and has a strongly polycyclic fundamental group. This is due to our usage of a result from the work of de Nicola and Yudin in \cite{Yudin2}, where they proved that a compact aspherical Sasakian manifold with nilpotent fundamental group is diffeomorphic to a nilmanifold with $H(1,n)$ as the associated Lie group. Therefore, we obtained a Vaisman analogue of this result in Theorem \ref{AsphericalVaisman}.

Another consequence is that every Oeljeklaus-Toma manifold, as originally defined in \cite{OT}, does not admit any Vaisman structure. The case where $J$ is left-invariant was previously solved by Kasuya in \cite{KasuyaOT}. For $J$ not necessarily left-invariant, the result is due to the main theorem, which imposes that every Vaisman solvmanifold contains a nilpotent sublattice in the Lie group $G$, as shown in Proposition \ref{OT}.

Finally, using a topological restriction of LCK Einstein-Weyl manifolds, we can apply the main theorem to show that no solvmanifold can admit a LCK Einstein-Weyl structure, which is done in section 5.

\subsection*{Acknowledgements} 
I would like to express my gratitude to my supervisor, Hisashi Kasuya, for his guidance and illuminating suggestions, which culminated in the results presented in this paper. I am also grateful to Viviana del Barco for her initial guidance in the study of Einstein-Weyl geometry.

\section{Preliminaries}
 
\subsection{LCK Structures}

Let $(M^n, J, \omega)$ be a Hermitian manifold with complex structure $J$ and fundamental form $\omega$ with real dimension $n \geq 4$. For this work we consider all manifolds to be connected.
\begin{Def}
\mbox{}
    \begin{itemize}
        \item If there exists a closed 1-form $\theta$ such that $d\omega = \theta \wedge \omega$, then the triple $(J, \omega, \theta)$ is called a Locally Conformally K\"ahler (LCK) structure on $M$.

        \item If $\theta$ is exact, then the Hermitian metric $g$ is conformal to a K\"ahler one. In this case we call the structure Globally Conformally K\"ahler (GCK).

        \item We call this structure Vaisman if $\theta$ is parallel with respect to the Levi-Civita connection, $\nabla^g \theta = 0$.
    \end{itemize}
\end{Def}
We implicitly assume that the LCK structures are not GCK. In this way we exclude the K\"ahler case, unless explicitly stated. This can be done because of the following fundamental fact.

\begin{Theo}[Vaisman \cite{Vaisman}]
    Let $(J, \omega, \theta)$ be an LCK structure on a compact manifold $M$. If the LCK structure is not GCK, then $(M, J)$ does not admit any K\"ahler metric.
\end{Theo}

Vaisman manifolds are closely related to Sasakian manifolds. 

\begin{Def}
    A Sasakian manifold is a Riemmanian manifold $(M, g)$ of odd dimension such that the Riemannian cone $(\R^{>0} \times M, d t^2+ t^2g)$ has a complex structure such that the metric is K\"ahler.
\end{Def}

In fact, Sasakian manifolds provide a source of examples of Vaisman manifolds by considering the $\Z$ action of a holomorphic homothety $h_\lambda (t, p) := ( \lambda t, p)$, for $\lambda > 0$, on $\R^{>0} \times M$. The quotient $\Z \backslash (\R^{>0} \times M)$ is naturally endowed with a Vaisman structure. 

By using the Structure Theorem in \cite{OV1, OVStructure}, Bazzoni, Marrero and Oprea obtained the following theorem which is key for our proof of Theorem \ref{MainTheorem}.

\begin{Theo}[\cite{Bazzoni2}]\label{Bazzoni}
     Let $M$ be a compact Vaisman manifold. Then, there exists a finite cover $p : S\times S^1 \to M$ , where $S$ is a compact Sasakian manifold, the deck group is isomorphic to $\Z_m$, for some $m > 0$, acts diagonally and by translations on the
     second factor. We have a diagram of fiber bundles
\[
\begin{tikzcd}
    S \arrow{r}{} \arrow{d}{} & S\times S^1 \arrow{r}{} \arrow{d}{p} & S^1 \arrow{d}{} \\
    S \arrow{r}{}  & M \arrow{r}{} & S^1  
\end{tikzcd}
\]
and M fibers over the circle $S^1/\Z_m$ with finite structure group $Z_m$.
\end{Theo}

\subsection{Einstein-Weyl LCK metrics}

For a given metric $g$ on a smooth manifold $M$ one can consider the conformal class $c$ defined by $g$, i.e., 
$c = \{e^{2f} g \ |\ f\in \mathcal{C}^{\infty}(M)\}$. We call the pair $(M, c)$ a conformal manifold. For a pair $(g, \theta)$ of a metric $g$ and a 1-form $\theta$, we can define an affine connection on $M$ by
\begin{equation*}
    \nabla_X Y = \nabla^g_X Y + \theta(X)Y + \theta(Y)X - g(X, Y)\theta^{\#_g}.
\end{equation*}
This affine connection satisfies $\nabla g = -2 \theta \otimes g$. One can show that for a conformal transformation $h := e^{2f} g$, by defining $\theta^h := \theta - 2df$, we get that $(h, \theta^h)$ defines the same $\nabla$.
We call $\nabla$ a Weyl structure on $M$ and $\theta^g$ the Lee form associated with $g$. We can define its curvature and Ricci curvature by
\begin{align*}
    R^\nabla(X, Y)Z &:= \nabla_{[X,Y]}Z - [\nabla_X, \nabla_Y]Z\\
    \Ric^{\nabla}(X,Y) &:= \Tr(Z \mapsto R^{\nabla}(X, Z)Y).
\end{align*}
By a choice of $g$ in $c$ we can define $\Scal_g^\nabla$ by the $g$-trace of the Ricci curvature $\Ric^\nabla$.

\begin{Def}
    Let $(M, c)$ be a conformal manifold with a Weyl structure $\nabla$, and denote $n := \dim M$. We say that $\nabla$ is Einstein-Weyl if
    $$\Ric^{\nabla} = \frac{1}{n} \Scal_g^{\nabla} g - \frac{1}{2}(n-2)d\theta,$$
    for one (hence for all) $g \in c$.
\end{Def}

An LCK structure $(\omega, \theta)$ on a complex manifold $(M, J)$ defines naturally a Weyl structure by considering the pair $(g, -2 \theta)$. If the $\nabla$ induced by the LCK structure in this way is Einstein-Weyl, we say that $M$ is an Einstein-Weyl LCK manifold. 

\begin{Theo}[\cite{Gauduchon1}]
    Let $(M, c)$ be a compact oriented conformal manifold of dimension $n \geq 3$ and let $\nabla$ be a Weyl structure on $(M, c)$. Then, there exists a metric $g \in c$, unique up to multiplication by a real number, such that the Lee form $\theta$ associated to $g$ in the Weyl structure satisfy $\delta^g \theta = 0$. 
\end{Theo}

The metric of the theorem above is called the Gauduchon metric.

\begin{Theo}[\cite{Gauduchon2}]
    Let $(M, c)$ and $\nabla$ be as in the previous theorem. Consider the Gauduchon metric $g$ and the associated Lee form $\theta$. If $\nabla$ is Einstein-Weyl and $\theta \neq 0$, then the Lee field $X = \theta^{\#_g}$, the $g$-dual of $\theta$, is a Killing vector field.
\end{Theo}

 If one considers that $\nabla^g \theta = \frac{1}{2}\mathcal{L}_{\theta^{\#_g}} g + \frac{1}{2} d \theta$ for any $1$-form $\theta$, then we obtain the following.

\begin{Cor}\label{Gauduchon}
    If $M$ is an Einstein-Weyl LCK manifold, then its Gauduchon metric $g$ is a Vaisman metric.
\end{Cor}

Denoting by $b_1(M)$ the first Betti number of $M$, we have the following result, which will be used to prove the final application of the main theorem of this work.  

\begin{Theo}[\cite{Verb}]\label{Verb}
    Let $M$ be a compact Einstein-Weyl LCK manifold. Assume $M$ is not K\"ahler. Then, $b_1(M) = 1$.
\end{Theo}

\subsection{Solvmanifolds}

\begin{Def}\label{HeinsbrgDef}
    We define the Heisenberg group $H(1,n)$ of dimension $2n+1$  to be the following Lie group
    $$H(1,n) := \left\{
     \begin{pmatrix}
 1&  x_1& x_2 & \cdots & x_n  & z \\
 &  1&  &  &  & y_1 \\
 &  &  \ddots&   &  & y_2 \\
 &  &  & \ddots &  & \vdots  \\
 &  &  &  & 1 & y_n \\
 &  &  &  &  & 1 \\
\end{pmatrix} \ \middle| \ x_i,y_i,z \in \R \right\}$$
It is a 2-step nilpotent Lie group admitting a lattice. It also has a natural left-invariant Sasakian structure. The group can also be described as $\R^{2n+1}$ with the group multiplication
$$(z, x_1, y_1, \dots, x_n, y_n) \cdot  (z', x_1', y_1', \dots, x_n', y_n') := (z + z' + \sum_{i = 1}^{n} x_i y_i'\, , x_1 + x_1', y_1+y_1', \dots, x_n + x_n', y_n+y_n').$$
\end{Def}

\begin{Def}
    Let $G$ be a solvable Lie group and $\g$ be its Lie algebra. Assume that $G$ is connected and simply-connected. 
    \begin{itemize}
        \item A manifold $M$ is called a solvmanifold if it is the quotient of $G$ by a lattice $\Gamma$, i.e., a discrete cocompact subgroup of $G$, acting on the left on $G$ by left multiplication. In this case we write $M = \Gamma \backslash G$. 

        \item  In the case where $G$ is nilpotent, $M$ is called a nilmanifold. 

        \item If $G = \R \times H(1,n)$, we call $M$ a Kodaira-Thurston manifold.

        \item We say that $G$ is completely solvable if $ad_X$ only has real eigeinvalues for all $X \in \g$. In this case, $M$ is called a completely solvable solvmanifold.
    \end{itemize}
\end{Def}

\begin{Def}
    We call a path-connected topological space $X$ aspherical if all its homotopy groups $\pi_p(X)$ of order $p \geq 2$ are trivial.
\end{Def}

Every solvmanifold is aspherical, since, by our definition, they are given by a smooth covering where the universal covering is homeomorphic to $\R^n$ (\cite[Preliminaries]{Raghunathan}).

\begin{Theo}[Rigidity Theorem]
\mbox{}
    \begin{itemize}
        \item \textup{\cite[Theorem 2.11]{Raghunathan}} Two nilmanifolds with isomorphic fundamental groups have isomorphic associated Lie groups.

        \item \textup{\cite{Mostow}} Two solvmanifolds with isomorphic fundamental group are diffeomorphic. 
    \end{itemize}
\end{Theo}
Recall that the nilradical $N$ of a solvable Lie group $G$ is the maximal normal connected nilpotent subgroup of $G$. 
\begin{Theo}[Mostow Structure Theorem \cite{Mostow}]
    Let $M = \Gamma \backslash G$ be a compact solvmanifold. Consider the nilradical $N$ of $G$. Then,
    \begin{enumerate}
        \item $N\Gamma$ is a closed subgroup of $G$;
        \item $N \cap \Gamma$ is a lattice in $N$;
        \item $N\Gamma \backslash G$ is a torus.
    \end{enumerate}
    Consequently, $M$ is fibered over a torus, with a nilmanifold as a fiber,
    \[
    (N \cap \Gamma)\backslash N \xrightarrow{} M \xrightarrow{} N\Gamma \backslash G
    \]
\end{Theo}

\begin{Def}
    We say that a group $\Gamma$ is strongly polycyclic if it admits a a sequence \[
    \Gamma = \Gamma_0 \supset \Gamma_1\supset \cdots\supset\Gamma_k = \{e\}
    \]
    of subgroups such that $\Gamma_i$ is normal in $\Gamma_{i-1}$ and $\Gamma_{i-1}/\Gamma_i$ is infinite cyclic.
\end{Def}

\begin{Example}
    We describe a family of examples given in \cite{AndradaOriglia} of solvmanifolds with left-invariant Vaisman structure that are not diffeomorphic to a nilmanifold. Let $\h(1,n)$ be the $2n+1$ dimensional Heisenberg nilpotent Lie algebra. Consider the basis $\{Z, X_1, Y_1, \dots, X_n, Y_n\}$ satisfying $[X_i, Y_i] = Z$, with all other brackets trivial. Let $a_1, \dots, a_n \in \R$ and define
\[
D:=\begin{pmatrix}
 0 &  &  &  &  & \\
 &  0& -a_1 &  &  &  \\
 &  a_1&  0&   &  &  \\
 &  &  & \ddots &  &   \\
 &  &  &  & 0 & -a_n \\
 &  &  &  & a_n & 0 \\
\end{pmatrix}
\]

Let $\g = \g_{a_1, \dots, a_n} := \R A \ltimes_D \h(1,n)$, where $A$ is just a symbol representing a generator. These are called the oscillator Lie algebras. One can show that if $\exists c \in \R\setminus \{0\}$ such that $(a_1, \dots, a_n) = c (b_1, \dots, b_n)$, then $\g_{a_1, \dots, a_n}$ and $\g_{b_1, \dots, b_n}$ are isomorphic. All these Lie algebras admits a Vaisman structure in the following manner. Define a complex structure $J$ by the relations $JA := Z, JX_i := Y_i$, and consider a metric $g$ given by declaring $\{A, Z, X_1, Y_1, \dots, X_n, Y_n\}$ an orthonormal basis of $\g$. Then, $g$ is an Hermitian metric by construction. Consider $\theta := A^*$ the dual of $A$ and $\omega$ the fundamental form. By \cite[Theorem 3.10]{AndradaOriglia} $(J, \omega, \theta)$ defines a Vaisman structure on $\g$. 

We consider now the associated Lie group $G = G_{a_1, \dots, a_n} := \R \ltimes_{\varphi} H(1,n)$, where we are considering $H(1,n)$ being described as $\R^{2n+1}$ and $\varphi(t) := e^{tD}$. Thefore,
\[
\varphi(t) = \begin{pmatrix}
 1 &  &  &  &  & \\
 &  \cos ta_1 & -\sin ta_1 &  &  &  \\
 &  \sin ta_1& \cos ta_1&   &  &  \\
 &  &  & \ddots &  &   \\
 &  &  &  & \cos ta_n & -\sin ta_n \\
 &  &  &  & \sin ta_n & \cos ta_n \\
\end{pmatrix}.
\]

In \cite{AndradaOriglia}, the authors takes $a_i \in \Z$ for all $i$. In this way, they consider a fixed lattice $\Gamma_k := \frac{1}{2k} \Z \times \dots \times \Z$ for $H(1,n)$ and construct three families of lattices in $G$ by
\begin{equation*}
    \Lambda_{k, \frac{\pi}{2}} := \frac{\pi}{2}\Z \ltimes_{\varphi} \Gamma_k, \qquad  \Lambda_{k, \pi} := \pi\Z \ltimes_{\varphi} \Gamma_k, \qquad \Lambda_{k, 2\pi} := 2\pi\Z \ltimes_{\varphi} \Gamma_k.
\end{equation*}
They also show that the first Betti number of the resulting solvmanifolds are all odd. Therefore, the LCK structure in $G$ passes down to the quotient to a non-GCK one, due to Vaisman's Theorem. If $D = 0$, then one recovers $\R\times H(1,n)$ with its canonical Vaisman structure. For $n = 1$, $G$ only depends of a single parameter $a \in \Z$. For $a = 0$, any Kodaira-Thurston manifold of $ \R \times H(1,1)$, with the complex structure described above, is called a primary Kodaira surface. For $a \neq 0$, any quotient by a non-nilpotent lattice is called a secondary Kodaira surface. It is known that a secondary Kodaira surface is a finite quotient of a primary one \cite{Hasegawa1}. 
\end{Example}

\section{Main Theorem and Corollaries}

\MainTheo

\begin{proof}
    Since $M$ is a compact Vaisman manifold, by Theorem \ref{Bazzoni} we have the following following diagram of fiber bundles
\[
\begin{tikzcd}
    S \arrow{r}{} \arrow{d}{} & S^1 \times S \arrow{r}{} \arrow{d}{q} & S^1 \arrow{d}{} \\
    S \arrow{r}{}  & M \arrow{r}{} & S^1  
\end{tikzcd}
\]
In particular, we have a fibration of $M$ by $S$ over $S^1$. Since $M$ is aspherical, using the long exact sequence of homotopy groups of fibrations, we obtain that $S$ is also aspherical and satisfy the following exact sequence on the fundamental groups

\[
\begin{tikzcd}
    1 \arrow{r}{} & \fg(S) \arrow{r}{} &\fg(M) \arrow{r}{} &\Z \arrow{r}{} & 1.
\end{tikzcd}
\]

Since $M$ is a solvmanifold we have that $\fg(M) = \Gamma$ and, by \cite[Theorem 4.28]{Raghunathan}, $\Gamma$ is a strongly polycyclic group. Since $\fg(S)$ is a subgroup of $\Gamma$, then it is also strongly polycyclic. Now, by \cite[Corollary 1.3.]{Kasuya1}, since $S$ is Sasakian, there is a nilpotent subgroup $\Lambda \leq \fg(S)$ of finite index. 

Let $u : \tilde{S} \to S$ be a covering of $S$ such that $\fg(\tilde{S}) = \Lambda$. Since there is a natural transitive action of $\fg(S)$ on the fibers of $u$, we know that the cardinality of any fiber $u^{-1}(\{x\})$ is given by the index of $\fg(\tilde{S})$ in $\fg(S)$ \cite[Theorem 10.9, Theorem 10.10]{Rotman}. Hence, $u$ is a finite covering, and we obtain that $\tilde{S}$ is compact. Endowing $\tilde{S}$ with the unique smooth structure such that $u$ is a smooth covering map, we obtain that $\tilde{S}$ is a compact aspherical Sasakian manifold with a nilpotent fundamental group. 

By \cite[Theorem 1.1]{Yudin2}, we obtain that $\tilde{S}$ is diffeomorphic to a Heisenberg nilmanifold $\Lambda \backslash H(1,n)$, where $\Lambda$ is being identified with a lattice in $H(1,n)$. Now, the composition $p : S^1 \times \tilde{S} \xrightarrow{\Id \times u} S^1 \times S \xrightarrow{q} M$ is also a finite smooth covering map of $M$. Since $( \Z \times \Lambda)\backslash (\R \times H(1,n))$ is diffeomorphic to $(\Z \backslash \R) \times (\Lambda \backslash H(1,n))  = S^1 \times \tilde{S}$ we conclude the theorem.
\end{proof}
\begin{Rem}
    In particular, the theorem agrees with the fact that secondary Kodaira surfaces are finite quotients of primary Kodaira surfaces.
\end{Rem}

Notice that in the proof the solvmanifold structure of $M$ was only used to obtain that $M$ is aspherical and $\pi_1(M)$ was strongly polycyclic. So more generally we have the following result.

\begin{Theo}\label{AsphericalVaisman}
    If $M$ is an aspherical compact Vaisman with strongly polycyclic fundamental group $\pi_1(M)$, then $M$ is a finite quotient of a Kodaira-Thurston manifold.
\end{Theo}

While it is clear that $ \Z \times \Lambda$ is isomorphic to a subgroup of the lattice $\Gamma$, in the proof of the theorem we only concluded that $N$ is given by the quotient of a nilpotent lattice in $\R\times H(1,n)$, however we can show that $\Z \times \Lambda$ is a \textbf{nilpotent lattice} of $G$, meaning $N$ is given by the quotient of the action by left multiplication of $\Z \times \Lambda$ in $G$. 

\begin{Lemma}
     Let $M = \Gamma \backslash G$ be a solvmanifold. Consider $q : N \to M$ a covering of $M$. Then, $\fg(N)$ is a discrete subgroup of $G$ acting by left translations on $G$. In particular, if $N$ is compact, then $\fg(N)$ is a sublattice of $\Gamma$.
\end{Lemma}

\begin{proof}

    Let $p : G \to N$ be the unique covering map such that the following diagram commutes
\[
\begin{tikzcd}
   G \arrow[rd]{}{p} \arrow[dd,swap]{}{\pi} &               \\
                         & N \arrow[ld]{}{q} \\
   M                    &   
\end{tikzcd}
\]
Let $\Aut(p)$ be the group of automorphisms of the covering $p$. Since $G$ is simply connected, we know that $\Aut(p)$ is isomorphic to $\fg(N)$. Now, for $\phi \in \Aut(p)$, we have $\pi \circ \phi = q \circ p\circ \phi = q \circ p = \pi$, thus $\Aut(p) \subset \Aut(\pi)$. Since $\Aut(\pi)$ is the set $\Gamma$ acting on the left on $G$, we obtain that $p$ is the quotient map of the left action of $\fg(N)$ on $G$. 
\end{proof}

\begin{Cor}
    Let $M = \Gamma \backslash G$ be a Vaisman solvmanifold. Then, $G$ admits a nilpotent sublattice.
\end{Cor}
\begin{proof}
    Consider $N$ the Kodaira-Thurston manifold given by Theorem \ref{MainTheorem}. Then, by the above lemma we have that $\fg(N) = \Z\times \Lambda$ is a sublattice of $G$. Since $\fg(N)$ is nilpotent, the corollary follows.
\end{proof}

\begin{Rem}
    We can find a finite covering by a Kodaira-Thurston manifold explicitly in the family of the oscillator Lie groups $G = G_{a_1, \dots, a_n}$ considered in section 2. For the lattice $\Lambda_{k, 2\pi}$, we have that $\varphi(2\pi t) = \Id$ for any $t \in \Z$, therefore $\Lambda_{k, 2\pi} = 2\pi\Z \times \Gamma_k$, which is isomorphic to a lattice in $\R \times H(1,n)$. By the Rigidity Theorem, we obtain that $\Lambda_{k, 2\pi}\backslash G$ is a Kodaira-Thurston manifold. 
    For the lattice $\Lambda_{k, \pi}$, if all $a_i$ are even, then $\Lambda_{k, \pi}$ is essentially $\Lambda_{k, 2 \pi}$. If at least one $a_i$ is odd, one can consider $\Lambda_{k, 2\pi}$ as a subgroup of $\Lambda_{k, \pi}$, and realize that $\Lambda_{k, 2\pi} \backslash \Lambda_{k, \pi} = \Z_2$. Defining a natural action of $\Z_2$ on $\Lambda_{k, 2\pi}\backslash G$, by $1 \cdot [t, x] := [t + \pi, x]$, it can be seen that the quotient by this action is isomorphic to $\Lambda_{k,\pi}\backslash G$. For $\Lambda_{k, \frac{1}{2}\pi}$ the process is analogous.
\end{Rem}

\ComplSolv

\begin{proof}
    By Saito's Rigidity Theorem \cite{Saito}, if a completely solvable Lie group $G$ admits a nilpotent lattice, then $G$ is nilpotent. Since $\Z \times \Lambda$ is such a lattice for $G$, we obtain that $(\Z \times \Lambda)\backslash G$ is a nilmanifold diffeomorphic to $(\Z \times \Lambda) \backslash (\R \times H(1,n))$. By using the Rigidity Theorem for nilmanifolds, the result follows.
\end{proof}

For the next application in this section, we consider the Oeljeklaus-Toma (OT) manifolds, constructed originally in \cite{OT}, however we use their solvmanifold description as showed in \cite{KasuyaOT}. 

One has $G = \R^s \ltimes_\phi (\R^s \times \mathbb{C}^t)$, where 
\[
\phi (a_1, \dots , a_s) := \diag (e^{a_1}, \dots, e^{a_s}, \exp \sum_{j =1}^s {c_{1j} a_j}, \dots, \exp \sum_{j =1}^s {c_{tj} a_j})
\]
and $C = (c_{ij})$ is a matrix of complex entries satisfying some properties from the construction of such manifolds (see  \cite{KasuyaOT} or \cite{Kanda} for more details). From the original construction of Oeljeklaus and Toma, Kasuya showed that one can construct a lattice of the type $\Gamma = U \ltimes_\phi \mathcal{O}$, where $U \leq \R^s$ is a lattice of $\R^s$ and $\mathcal{O} \leq \R^s \times \mathbb{C}^t$ a lattice of $\R^s \times \mathbb{C}^t$. The quotient $M = \Gamma\backslash G$ is a solvmanifold which is not diffeomorphic to a nilmanifold called a OT manifold, and it carries a natural left-invariant non-Vaisman structure. In fact, Kasuya showed that, for $J$ left-invariant, $M$ cannot admit any Vaisman structure. We extend this result for any $J$ as follows.

\begin{Prop}\label{OT}
    Let $M$ be a OT manifold as above. Then, $M$ does not admit any Vaisman structure.
\end{Prop}
\begin{proof}
    First observe that $N :=\R^s \times \mathbb{C}^t$ is the nilradical of $G$. Now, suppose that there exists a $\Delta \leq U \ltimes_\phi \mathcal{O}$ nilpotent sublattice. Notice that $\Delta \cap \mathcal{O}$ is a lattice of $N$ by the Mostow Structure Theorem, hence $\Delta \cap \mathcal{O} \neq \{0\}$. Given $(a,x) \in G$, we know that $(a,x)^{-1}= (-a, -\phi(a)^{-1}x)$. For a fixed $(a,x) \in G$, define $I(b,y) := (a,x)(b,y)(a,x)^{-1}(b,y)^{-1} = (0, (\phi(a)-\Id)y - (\phi(b)-\Id)x)$. Thus, by considering $(a, x) \in \Delta$ and $(0, y) \in \Delta \cap \mathcal{O}$, we get $I^{n}(0,y) = (0, (\phi(a)-\Id)^n y)$. Using the fact that the central series of $\Delta$ has finite length, we obtain that there exists a $n \in \N$ such that $(\phi(a) - \Id)^n = 0$ for all $a$ in the projection of $\Delta$ on $U$. However, $\phi(a)$ is a diagonal matrix, hence such $n$ can only exists if $a = 0$. This means that $\Delta \leq \mathcal{O} \leq N$. By Mostow Structure Theorem $(N\Delta) \backslash G$ has to be a torus, which is impossible since $\Delta$ does not act on $\R^s$. 
\end{proof}

\section{The case LCK with potential}

Let $M$ be an LCK manifold with LCK structure $(J, \omega, \theta)$. Defining $\dtheta := d - \theta\wedge$, if there exists a positive smooth function $\psi : M \to \R^+$ such that $\omega = \dtheta \dtheta^c \psi$ we call $(J, \omega, \theta)$ an LCK structure with potential, where $\dtheta^c := J \dtheta J^{-1}$. The Vaisman structure is a particular case by considering the constant function $\psi = 1$. These structures are abundant in a sense, since they are stable under small deformations \cite{OV3}. Therefore, any small deformation of a Vaisman manifold gives an example of LCK structure with potential. 

In the proof of Theorem \ref{MainTheorem}, the LCK structure was invoked only to obtain a covering of the manifold with good properties. In fact, if $M$ admits a finite covering as in Theorem \ref{Bazzoni}, then the results of Theorem \ref{MainTheorem} are still true regardless if $M$ admits a structure $(J,\omega, \theta)$. Therefore, by deforming the initial complex structure $J$ we can actually extend the main result to LCK structures with potential. This can be done for $\dim M \geq 6$ due to a result of Ornea and Verbitsky in \cite{OV2}. In dimension 4, we check directly using the classification due to Hasegawa below.

\begin{Theo}[\cite{Hasegawa1}]

    Let $(M, J)$ be a complex manifold of real dimension $4$. Then, $M$ is diffeomorphic to a solvmanifold if, and only if, $(M ,J)$ is biholomorphic to one of the following classes of complex surfaces:
    \begin{enumerate}
        \item Complex tori;
        \item Hyperelliptic surfaces;
        \item Inoue surfaces of type $S_0$, $S_{+}$ and $S_-$;
        \item Primary Kodaira surfaces;
        \item Secondary Kodaira surfaces.
    \end{enumerate}
    Each of these complex surfaces is equipped with a left-invariant complex structure as a solvmanifold. In other words, every complex structure in these classes arise from a canonically defined left-invariant complex structure in the associated Lie group.
\end{Theo}

Therefore, we only need to verify the existence of LCK structures with potential in each of these complex surfaces. We collect the following important results:

\begin{itemize}
    \item Complex tori and Hyperelliptic surfaces admits a K\"ahler structure, thus cannot have a non-GCK structure.

    \item Primary Kodaira surfaces are Kodaira-Thurston nilmanifolds. In particular, it admits a Vaisman structure.
    
    \item Secondary Kodaira surfaces are finite quotients of Primary Kodaira surfaces. In particular, they admit a Vaisman structure.

    \item The Inoue surfaces of types $S_0$, $S_{+}$ and $S_-$ do not admit LCK structures with potential. \cite[Corollary 4.13]{Otiman}
\end{itemize}
We emphasize that given any of these complex surfaces $(M, J)$ above, the complex structure uniquely defines the surface inside its class. This means that if one tries to endow $M$ with another complex structure $J'$, since $M$ still has the same diffeomorphism type, $(M, J')$ is still in the same class. In particular, $J'$ is also left-invariant. This ensures us that none of the complex surfaces above admits an LCK with potential $(J, \omega, \theta)$ for a different $J$ than the initial one.

\begin{Cor}
    Let $M = \Gamma \backslash G$ be a solvmanifold. If $M$ admits an LCK structure with potential $(J, \omega, \theta)$, 
    then $M$ is diffeomorphic to a finite quotient of a Kodaira-Thurston manifold $N = (\Z \times \Lambda)\backslash(\R \times H_{2n+1})$, for some $\Lambda \leq \Gamma$. 

    More precisely, we have a smooth finite covering $p : N \to M$ in the following diagram.
\[
\begin{tikzcd}[ampersand replacement=\&]
G \arrow[rd] \arrow[rdd, "\pi"', bend right] \&                                               \& {\mathbb{R}\times H_{2n+1}} \arrow[ld] \\
                                             \&  N \arrow[d, "p"] \&                                           \\
                                             \& \Gamma \backslash G                                    \&                                          
\end{tikzcd}
\]
    
\end{Cor}
\begin{proof}
    Suppose $\dim M \geq 6$. Since $M$ admits an LCK structure with potential $(J, \omega, \theta)$, by \cite[Theorem 7.3]{OV2} $M$ admits another complex structure $J'$ such that there exists a Vaisman structure $(J',\omega',\theta')$ in $M$. Hence, the conclusion of theorem \ref{MainTheorem} follows.

    For $\dim M = 4$, we consider the classification above on complex surfaces diffeomorphic to solvmanifolds. By the results collected above, we saw that the only possible complex surfaces that admit an LCK metric with potential are the primary and secondary Kodaira surfaces, which satisfy the theorem trivially.
\end{proof}

\section{Application to Einstein-Weyl LCK case}

For the final application of Theorem \ref{MainTheorem}, we recall the following two results.

\begin{Theo}[\cite{BensonGordon, Hasegawa2}]\label{Hasegawa}
    Let $M$ be nilmanifold. If $M$ admits a K\"ahlerian structure, then $M$ is diffeomorphic to a torus.
\end{Theo}

\begin{Theo}[\cite{Nomizu}]\label{Nomizu}
    Let $M = \Gamma \backslash G$ be a nilmanifold. Then $H^p(M) = H^p(\mathfrak{g})$ for all $p > 0$, where $\mathfrak{g}$ is the Lie algebra of $G$.
\end{Theo}

\begin{Prop}
    Let $M$ be a solvmanifold. Then, $M$ does not admit any Einstein-Weyl LCK structure.
\end{Prop}
\begin{proof}
    Suppose $(J, g_0, \theta)$ is an Einstein-Weyl LCK structure on $M$. Then, by considering the Gauduchon metric $g$, $(M, J, g)$ is Vaisman by Corollary \ref{Gauduchon}. Applying the main theorem we obtain that there exists a nilmanifold $N$ from $\R \times H(1,n)$ which is a finite covering of $M$. We can lift the structure $(J, g, \theta)$ to $N$ obtaining an Einstein-Weyl LCK structure on $N$. Since Vaisman is a local property, $(N, J, g)$ is also a compact Vaisman manifold.

    Now, if $N$ is a K\"ahler manifold, by Theorem \ref{Hasegawa}, we have that $N$ is diffeomorphic to a torus. But then, by the Rigidity Theorem, this would imply that $\R \times H(1,n)$ is isomorphic as a Lie group to $\R^n$, which is impossible. Hence $N$ is not K\"ahler. Finally, we can apply Theorem \ref{Verb} to obtain that $b_1 (N) = 1$, however by Theorem \ref{Nomizu} $H^1(N) = H^1(\R \times \mathfrak{h})$ which implies that $b_1(N) = 2n + 1$, contradiction, since $n \geq 1$. Thus, $(J, g_0, \theta)$ cannot be Einstein-Weyl.
\end{proof}

\printbibliography

@article {BazzoniMain,
    AUTHOR = {Bazzoni, Giovanni},
     TITLE = {Vaisman nilmanifolds},
   JOURNAL = {Bull. Lond. Math. Soc.},
  FJOURNAL = {Bulletin of the London Mathematical Society},
    VOLUME = {49},
      YEAR = {2017},
    NUMBER = {5},
     PAGES = {824--830},
}

@article {Kasuya1,
    AUTHOR = {Kasuya, Hisashi},
     TITLE = {Cohomologies of {S}asakian groups and {S}asakian
              solvmanifolds},
   JOURNAL = {Ann. Mat. Pura Appl. (4)},
  FJOURNAL = {Annali di Matematica Pura ed Applicata. Series IV},
    VOLUME = {195},
      YEAR = {2016},
    NUMBER = {5},
     PAGES = {1713--1719},
}

@article {Sawai1,
    AUTHOR = {Sawai, Hiroshi},
     TITLE = {Locally conformal {K}\"ahler structures on compact
              nilmanifolds with left-invariant complex structures},
   JOURNAL = {Geom. Dedicata},
  FJOURNAL = {Geometriae Dedicata},
    VOLUME = {125},
      YEAR = {2007},
     PAGES = {93--101},
}

@article {Sawai2,
    AUTHOR = {Sawai, Hiroshi},
     TITLE = {Structure theorem for {V}aisman completely solvable
              solvmanifolds},
   JOURNAL = {J. Geom. Phys.},
  FJOURNAL = {Journal of Geometry and Physics},
    VOLUME = {114},
      YEAR = {2017},
     PAGES = {581--586},
}

@article {Yudin2,
    AUTHOR = {de Nicola, Antonio and Yudin, Ivan},
     TITLE = {Nilpotent aspherical {S}asakian manifolds},
   JOURNAL = {Int. Math. Res. Not. IMRN},
  FJOURNAL = {International Mathematics Research Notices. IMRN},
      YEAR = {2024},
    NUMBER = {15},
     PAGES = {11221--11238},
}

@article {UgarteMain,
    AUTHOR = {Ugarte, Luis},
     TITLE = {Hermitian structures on six-dimensional nilmanifolds},
   JOURNAL = {Transform. Groups},
  FJOURNAL = {Transformation Groups},
    VOLUME = {12},
      YEAR = {2007},
    NUMBER = {1},
     PAGES = {175--202},
}

@book {OrneaVerbitsky,
    AUTHOR = {Ornea, Liviu and Verbitsky, Misha},
     TITLE = {Principles of locally conformally {K}\"ahler geometry},
    SERIES = {Progress in Mathematics},
    VOLUME = {354},
 PUBLISHER = {Birkh\"auser/Springer, Cham},
      YEAR = {2024},
     PAGES = {xxi+736},
}

@article {OV1,
    AUTHOR = {Ornea, Liviu and Verbitsky, Misha},
     TITLE = {L{CK} rank of locally conformally {K}\"ahler manifolds with
              potential},
   JOURNAL = {J. Geom. Phys.},
  FJOURNAL = {Journal of Geometry and Physics},
    VOLUME = {107},
      YEAR = {2016},
     PAGES = {92--98},
}

@article {OV2,
    AUTHOR = {Ornea, Liviu and Verbitsky, Misha},
     TITLE = {Lee classes on {LCK} manifolds with potential},
   JOURNAL = {Tohoku Math. J. (2)},
  FJOURNAL = {The Tohoku Mathematical Journal. Second Series},
    VOLUME = {76},
      YEAR = {2024},
    NUMBER = {1},
     PAGES = {105--125},
}

@article {Bazzoni2,
    AUTHOR = {Bazzoni, G. and Marrero, J. C. and Oprea, J.},
     TITLE = {A splitting theorem for compact {V}aisman manifolds},
   JOURNAL = {Rend. Semin. Mat. Univ. Politec. Torino},
  FJOURNAL = {Rendiconti del Seminario Matematico. Universit\`a{} e
              Politecnico Torino},
    VOLUME = {74},
      YEAR = {2016},
    NUMBER = {1},
     PAGES = {21--29},
}

@article {Saito,
    AUTHOR = {Saito, Masahiko},
     TITLE = {Sur certains groupes de {L}ie r\'esolubles. {II}},
   JOURNAL = {Sci. Papers College Gen. Ed. Univ. Tokyo},
  FJOURNAL = {Scientific Papers of the College of General Education.
              University of Tokyo},
    VOLUME = {7},
      YEAR = {1957},
     PAGES = {157--168},
}

@book {Raghunathan,
    AUTHOR = {Raghunathan, M. S.},
     TITLE = {Discrete subgroups of {L}ie groups},
    SERIES = {Ergebnisse der Mathematik und ihrer Grenzgebiete [Results in
              Mathematics and Related Areas]},
    VOLUME = {Band 68},
 PUBLISHER = {Springer-Verlag, New York-Heidelberg},
      YEAR = {1972},
     PAGES = {ix+227},
}

@article {AndradaOriglia,
    AUTHOR = {Andrada, A. and Origlia, M.},
     TITLE = {Vaisman solvmanifolds and relations with other geometric
              structures},
   JOURNAL = {Asian J. Math.},
  FJOURNAL = {Asian Journal of Mathematics},
    VOLUME = {24},
      YEAR = {2020},
    NUMBER = {1},
     PAGES = {117--145},
}

@article {Hasegawa1,
    AUTHOR = {Hasegawa, Keizo},
     TITLE = {Complex and {K}\"ahler structures on compact solvmanifolds},
      NOTE = {Conference on Symplectic Topology},
   JOURNAL = {J. Symplectic Geom.},
  FJOURNAL = {The Journal of Symplectic Geometry},
    VOLUME = {3},
      YEAR = {2005},
    NUMBER = {4},
     PAGES = {749--767},
}

@article {Hasegawa,
    AUTHOR = {Hasegawa, Keizo},
     TITLE = {A note on compact solvmanifolds with {K}\"ahler structures},
   JOURNAL = {Osaka J. Math.},
  FJOURNAL = {Osaka Journal of Mathematics},
    VOLUME = {43},
      YEAR = {2006},
    NUMBER = {1},
     PAGES = {131--135},
}

@article {Hasegawa2,
    AUTHOR = {Hasegawa, Keizo},
     TITLE = {Minimal models of nilmanifolds},
   JOURNAL = {Proc. Amer. Math. Soc.},
  FJOURNAL = {Proceedings of the American Mathematical Society},
    VOLUME = {106},
      YEAR = {1989},
    NUMBER = {1},
     PAGES = {65--71},
}

@article {BensonGordon,
    AUTHOR = {Benson, Chal and Gordon, Carolyn S.},
     TITLE = {K\"ahler and symplectic structures on nilmanifolds},
   JOURNAL = {Topology},
  FJOURNAL = {Topology. An International Journal of Mathematics},
    VOLUME = {27},
      YEAR = {1988},
    NUMBER = {4},
     PAGES = {513--518},
}

@article {Gauduchon2,
    AUTHOR = {Gauduchon, Paul},
     TITLE = {Structures de {W}eyl-{E}instein, espaces de twisteurs et
              vari\'et\'es de type {$S^1\times S^3$}},
   JOURNAL = {J. Reine Angew. Math.},
  FJOURNAL = {Journal f\"ur die Reine und Angewandte Mathematik. [Crelle's
              Journal]},
    VOLUME = {469},
      YEAR = {1995},
     PAGES = {1--50},
}

@article {Gauduchon1,
    AUTHOR = {Gauduchon, Paul},
     TITLE = {La {$1$}-forme de torsion d'une vari\'et\'e{} hermitienne
              compacte},
   JOURNAL = {Math. Ann.},
  FJOURNAL = {Mathematische Annalen},
    VOLUME = {267},
      YEAR = {1984},
    NUMBER = {4},
     PAGES = {495--518},
}

@article {OVStructure,
    AUTHOR = {Ornea, Liviu and Verbitsky, Misha},
     TITLE = {Structure theorem for compact {V}aisman manifolds},
   JOURNAL = {Math. Res. Lett.},
  FJOURNAL = {Mathematical Research Letters},
    VOLUME = {10},
      YEAR = {2003},
    NUMBER = {5-6},
     PAGES = {799--805},
}

@article {Otiman,
    AUTHOR = {Otiman, Alexandra},
     TITLE = {Morse-{N}ovikov cohomology of locally conformally {K}\"ahler
              surfaces},
   JOURNAL = {Math. Z.},
  FJOURNAL = {Mathematische Zeitschrift},
    VOLUME = {289},
      YEAR = {2018},
    NUMBER = {1-2},
     PAGES = {605--628},
}

@article {Verb,
    AUTHOR = {Verbitsky, M. S.},
     TITLE = {Theorems on the vanishing of cohomology for locally
              conformally hyper-{K}\"ahler manifolds},
   JOURNAL = {Tr. Mat. Inst. Steklova},
  FJOURNAL = {Trudy Matematicheskogo Instituta Imeni V. A. Steklova},
    VOLUME = {246},
      YEAR = {2004},
     PAGES = {64--91},
}

@article {OV3,
    AUTHOR = {Ornea, Liviu and Verbitsky, Misha},
     TITLE = {Locally conformal {K}\"ahler manifolds with potential},
   JOURNAL = {Math. Ann.},
  FJOURNAL = {Mathematische Annalen},
    VOLUME = {348},
      YEAR = {2010},
    NUMBER = {1},
     PAGES = {25--33},
}

@article {Nomizu,
    AUTHOR = {Nomizu, Katsumi},
     TITLE = {On the cohomology of compact homogeneous spaces of nilpotent
              {L}ie groups},
   JOURNAL = {Ann. of Math. (2)},
  FJOURNAL = {Annals of Mathematics. Second Series},
    VOLUME = {59},
      YEAR = {1954},
     PAGES = {531--538},
}

@article {KasuyaOT,
    AUTHOR = {Kasuya, Hisashi},
     TITLE = {Vaisman metrics on solvmanifolds and {O}eljeklaus-{T}oma
              manifolds},
   JOURNAL = {Bull. Lond. Math. Soc.},
  FJOURNAL = {Bulletin of the London Mathematical Society},
    VOLUME = {45},
      YEAR = {2013},
    NUMBER = {1},
     PAGES = {15--26},
}

@article {OT,
    AUTHOR = {Oeljeklaus, Karl and Toma, Matei},
     TITLE = {Non-{K}\"ahler compact complex manifolds associated to number
              fields},
   JOURNAL = {Ann. Inst. Fourier (Grenoble)},
  FJOURNAL = {Universit\'e{} de Grenoble. Annales de l'Institut Fourier},
    VOLUME = {55},
      YEAR = {2005},
    NUMBER = {1},
     PAGES = {161--171},
}

@article {Vaisman,
    AUTHOR = {Vaisman, Izu},
     TITLE = {On locally and globally conformal {K}\"ahler manifolds},
   JOURNAL = {Trans. Amer. Math. Soc.},
  FJOURNAL = {Transactions of the American Mathematical Society},
    VOLUME = {262},
      YEAR = {1980},
    NUMBER = {2},
     PAGES = {533--542},
}

@article {Mostow,
    AUTHOR = {Mostow, G. D.},
     TITLE = {Factor spaces of solvable groups},
   JOURNAL = {Ann. of Math. (2)},
  FJOURNAL = {Annals of Mathematics. Second Series},
    VOLUME = {60},
      YEAR = {1954},
     PAGES = {1--27},
}

@misc{Kanda,
      title={A characterization of Oeljeklaus-Toma manifolds in locally conformally K\"{a}hler geometry}, 
      author={Shuho Kanda},
      year={2025},
      eprint={2502.12500},
      archivePrefix={arXiv},
      primaryClass={math.DG},
}

@book {Rotman,
    AUTHOR = {Rotman, Joseph J.},
     TITLE = {An introduction to algebraic topology},
    SERIES = {Graduate Texts in Mathematics},
    VOLUME = {119},
 PUBLISHER = {Springer-Verlag, New York},
      YEAR = {1988},
     PAGES = {xiv+433},
}
\end{document}